\documentclass[12pt]{amsart}

\usepackage{geometry}                % See geometry.pdf to learn the layout options. There are lots.
\usepackage{graphicx}
\usepackage{amssymb}
\usepackage{float}
\usepackage{hyperref} 
\usepackage{bbm}
\usepackage{amsmath,amscd}
\usepackage{algorithm,algorithmic}
\usepackage[english]{babel}
\usepackage[latin1]{inputenc}
\usepackage{times}
\usepackage[T1]{fontenc}
\usepackage{latexsym}
\usepackage{graphics}
\usepackage{color}
\usepackage{ulem}

\newtheorem{theorem}{Theorem}[section]

\newtheorem{fact}[theorem]{Fact}
\newtheorem{corollary}[theorem]{Corollary}

\newtheorem{proposition}[theorem]{Proposition}

\theoremstyle{definition}
\newtheorem{definition}[theorem]{Definition}

\newtheorem{example}[theorem]{Example}

\theoremstyle{remark}
\newtheorem{remark}[theorem]{Remark}

\numberwithin{equation}{section}

\def\Inf{\operatornamewithlimits{inf\vphantom{p}}}

\newcommand{\cE}{\mathcal{E}}
\newcommand{\cT}{\mathcal{T}}

\newcommand{\cN}{\mathcal{N}}

\newcommand{\cP}{\mathcal{P}}

\newcommand{\bP}{\mathbb{P}}

\newcommand{\bE}{\mathbb{E}}

\newcommand{\Z}{\mathbb{Z}}

\newcommand{\R}{\mathbb{R}}

\newcommand{\la}{\lambda }

\newcommand{\si}{\sigma }

\newcommand{\ga}{\gamma }
\newcommand{\Ga}{\Gamma }

\newcommand{\ones}{\mathbbm{1}}
\newcommand{\one}{\mathbf{1}}

\newcommand{\diam}{\operatorname{Diam}}
\newcommand{\rank}{\operatorname{rank}}

\newcommand{\Mod}{\operatorname{Mod}}

\newcommand{\defeq}{\mathrel{\mathop:}=}
\newcommand{\cCeff}{\mathop{\mathcal{C}_{\eff}}}
\newcommand{\cReff}{\mathop{\mathcal{R}_{\eff}}}

\newcommand{\eff}{\textrm{eff}}

\renewcommand{\Pr}{\mathbb{P}}

\newcommand{\Adm}{\operatorname{Adm}}

\newcommand{\etol}{\epsilon_{\text{tol}}}

\newcommand{\bi}{\begin{itemize}}
\newcommand{\ei}{\end{itemize}}

\newcommand{\vt}{\vspace{0.2cm}}

\begin{document}

\title[Minimal subfamilies]{Minimal subfamilies and the probabilistic interpretation for modulus on graphs}

\author[Albin]{Nathan Albin}
\author[Poggi-Corradini]{Pietro Poggi-Corradini}

\address{Department of Mathematics\\Kansas State University\\ Manhattan, KS 66506}

\thanks{Research partially supported by NSF grants  DMS-1201427 and DMS-1515810}

\dedicatory{On the occasion of David Minda's retirement.}

\begin{abstract}
The notion of $p$-modulus of a family of objects on a graph is a measure of the richness of such families.
We develop the notion of minimal subfamilies using the method of  Lagrangian duality for  $p$-modulus. We show that minimal subfamilies have at most $|E|$ elements and that these elements carry a weight related to their ``importance'' in relation to the corresponding $p$-modulus problem. When $p=2$, this measure of importance is in fact a probability measure and modulus can be thought as trying to minimize the expected overlap in the family.
\end{abstract}

\maketitle

\section{Introduction}

 This study is part of a larger project to enhance both the theoretical understanding of the ways in which diseases spread in an interconnected network of individuals or sub-populations, and the computational tools available to researchers interested in modeling, simulating, and predicting the behavior of epidemics. Interconnectedness plays a key role in the spread of disease in a population, wherein an infection is transmitted from one individual to another and so on through a complex tangle of individual interactions. Individuals in frequent contact with many others are more likely to both contract and to transmit a disease than individuals with few connections, and well-connected populations are more susceptible to epidemic than sparsely connected populations. Thus, understanding and quantifying the degree to which an individual is connected to others is crucial to the study of diseases, their escalation to epidemic status, and possible mitigation strategies.

Classical methods for quantifying interconnectedness are based on graph-theoretic concepts such as effective conductance, minimum cut, and shortest paths. These seemingly disparate concepts turn out to be special cases of a very general object, known in the context of classical analysis as $p$-modulus.  When applied to the context of graphs, $p$-modulus generalizes each of these classical methods for quantifying interconnectedness.

The strength of $p$-modulus lies in the fact that it is much more flexible than the properties it generalizes.  Whereas the classical quantities are all related to the single basic question ``How easily will a disease be able to travel from individual A to individual B?'', the $p$-modulus can quantify the answers to a wide variety of more detailed questions such as ``How easily will a disease be able to travel from individual A to individual B and then on to individual C?'' and ``What are the most important disease transmission pathways connecting two particular sub-populations?''

The study of epidemic phenomena affects the world at large, as can be witnessed from recent disease outbreaks on a global scale. Effective tools to model, predict and mitigate these pandemics can be of practical use to health authorities and scientists, see \cite{gassp:ieee2015}.
Moreover, in view of the general nature of $p$-modulus as a method, its applications go beyond those related to epidemics. Networks are ubiquitous in science and engineering, and the results of this project are applicable in many other settings, \cite{spcsa:jcam2016}.

In this paper we continue our investigation of the mathematical concept of p-modulus on networks, that began in \cite{adsgpc} and \cite{abppcw:ecgd2015}. Here we focus on the analysis of theoretical properties and the development of numerical algorithms from the point of view of Lagrangian duality. In particular we introduce and study the notion of minimal subfamilies. Software based on our techniques has been developed by the first named author, and has already been used in applications, e.g., in \cite{spcsa:jcam2016}. 

\subsection{Modulus on graphs}

The investigation of  $p$-modulus  on graphs comprises three main directions: the theoretical analysis of the mathematical concept, the development of efficient computational
algorithms for its evaluation, and the application of both theory and
algorithms to epidemic modeling and prediction.

The classical theory of $p$-modulus is a continuum theory and was
developed originally in the field of complex
analysis, see the comment on p.~81 of \cite{ahlfors1973}.  In short, $p$-modulus
provides a method for quantifying the richness of families of curves. Heuristically,
families with many short curves have a larger modulus than families with
fewer longer curves.  Although the concept of a discrete modulus is not
new (see, e.g., \cite{duffin:jmaa1962,schramm:ijm1993}), it is less
well understood than in the continuum setting and has not been
considered in the context presented in this paper.

One of the strengths of modulus on graphs is its great versatility and generality.
We will work with weighted and possibly directed finite graphs. So $G= (V,E,\si)$ is a {\it graph} where the vertex-set $V$ and the edge-set $E$ are finite. To every edge $e\in E$ corresponds a weight $0<\si(e)<\infty$.

We will consider {\it families} $\Ga$ of objects $\ga$ on $G$, such as
families of walks, families of trees, etc.
\begin{definition}\label{def:matrixn}
  We will consider families of objects $\Ga$ such that for each
  $\ga\in\Ga$ we can associate a function $\cN(\ga,\cdot):
  E\rightarrow \R_{\ge 0}$, i.e. a vector $\cN(\ga,\cdot)\in\R_{\ge 0}^E$, that
  measures the {\it usage of edge $e$ by $\ga$}.  We assume that each
  $\gamma\in\Gamma$ has positive usage on at least one edge.
\end{definition}

For example:
\bi
\item A walk $\ga=x_0\ e_1\ x_1\ \cdots\ e_n\ x_n$ is associated to the function $\cN(\ga,e)=$ number times $\ga$ traverses $e$. In this case $\cN(\ga,\cdot)\in\Z_{\ge 0}^E$.
\vt
\item A set $T\subset E$ is associated to the function $\cN(T,e)=\ones_T(e)=1$ if $e\in T$ and $0$ otherwise.
Here, $\cN(\ga,\cdot)\in\{0,1\}^E$.
\vt
\item A flow $f$ can be associated with the function $\cN(f,e)=|f(e)|$. Therefore, $\cN(\ga,\cdot)\in \R_{\ge 0}^E$.
\ei
In other words, to each family of objects $\Ga$ we associate a $|\Ga|\times|E|$ {\it usage matrix} $\cN$ such that each row of $\cN$ corresponds to an object $\ga\in\Ga$ and records the usage of edge $e$ by $\ga$. Note, however, that the families $\Ga$ under consideration may very well be infinite.

Given a density  $\rho\in\R_{\ge 0}^E$, the weight $\rho(e)$  represents the {\it cost of using} edge $e$.
In particular, given an object $\ga\in\Ga$, we let
\[
\ell_\rho(\ga):=\sum_{e\in E} \cN(\ga,e)\rho(e) = (\cN \rho)(\ga),
\]
represent the {\it total usage cost}  for $\ga$.

A density $\rho\in\R_{\ge 0}^E$ is {\it admissible for the family $\Ga$},  if
\[
\ell_\rho(\ga) \geq 1\qquad\forall \ga\in\Ga;
\]
or equivalently, if
\[
\ell_\rho(\Ga)\defeq\inf_{\ga\in\Ga}\ell_\rho(\ga) \geq 1;
\]
in matrix notations, we want
\[
\cN \rho \geq {\bf 1}.
\]
Let 
\begin{equation}\label{eq:Adm}
\Adm(\Ga)=\left\{\rho\in\R_{\ge 0}^E: \cN \rho\geq {\bf 1}\right\}
\end{equation}
be the set of admissible densities.

For $1\leq p<\infty$,  the {\it $p$-energy} of $\rho$ is defined as
\begin{equation*}
  \cE_{p,\si}(\rho) \defeq \sum_{e\in E} \si(e)|\rho(e)|^p.
\end{equation*}

For $p=\infty$, we also define the {\it $\infty$-energy} as
\[\cE_\infty(\rho) \defeq
\lim_{p\to\infty}\left(\cE_{p,\si}(\rho)\right)^{\frac{1}{p}} = \max_{e\in
  E}|\rho(e)|\]

\begin{definition}
Given a graph $G= (V,E,\si)$ and a family $\Ga\subset\R_{\ge 0}^E$,
if $1\le
p\le\infty$,  the {\it $p$-modulus} of $\Gamma$ is:
\[ \Mod_{p,\si}(\Ga)\defeq \inf_{\rho\in \Adm(\Ga)}\cE_{p,\si}(\rho)\]
\end{definition}

Equivalently, $p$-modulus is the following optimization problem
\begin{equation*}
  \begin{split}
    \text{minimize} &\qquad \cE_{p,\si}(\rho) \\
    \text{subject to} &\qquad  \cN \rho \geq {\bf 1}
  \end{split}
\end{equation*}
where each object $\ga$ adds one constraint to be satisfied.

\begin{remark}\label{rem:remarks}
\bi
\item[(a)] When $\rho_0\equiv 1$ is the constant density equal to $1$,
  we drop the subscript and write
  $\ell(\ga)\defeq\ell_{\rho_0}(\ga)$. If $\ga$ is a walk, then
  $\ell(\ga)$ simply counts the number of hops that the walk $\ga$
  makes.
\item[(b)] If $\Ga\subset\Ga'$, then $\Adm(\Ga')\subset \Adm(\Ga)$, so
  $\Mod_{p,\si}(\Ga)\leq \Mod_{p,\si}(\Ga')$, for all $1\leq p\leq
  \infty$. This is known as {\it monotonicity} of modulus.
\item[(c)] For $1<p<\infty$ a unique extremal density $\rho^*$ always
  exists and satisfies $0\leq \rho^*\leq \cN_{\text{min}}^{-1}$, where
  $\cN_{\text{min}}$ is the smallest nonzero entry of $\cN$. Existence
  and uniqueness follows by compactness and strict convexity of
  $\cE_{p,\si}$, see also Lemma 2.2 of \cite{abppcw:ecgd2015}.  The
  upper bound on $\rho^*$ follows from the fact that each row of $\cN$
  contains at least one nonzero entry, which must be at least as large
  as $\cN_{\text{min}}$.  In the special case when $\cN$ is integer
  valued, the upper bound can be taken to be $1$.\ei
\end{remark}

\subsection{Connection to classical quantities}

The concept of $p$-modulus generalizes known classical ways of
measuring the richness of a family of walks.  Given two vertices $s$
and $t$ in $V$, we define the {\it connecting family} $\Ga(s,t)$ to be
the family of all walks in $G$ that start at $s$ and end at $t$.  Also
a subset $C\subset E$ is called a {\it cut} for a family of walks
$\Ga$ if for every $\ga\in\Ga$, there is $e\in C$ such that
$\cN(\ga,e)\neq 0$.  Moreover, the size of a cut is measured by
$|C|\defeq \sum_{e\in C}\si(e)$.  When the graph $G$ is undirected, it
can be thought of as an electrical network with edge-conductances given
by the weights $\si$, see the monograph \cite{doyle-snell1984}. In
this case, given two vertices $s$ and $t$ in $V$, we write
$\cCeff(s,t)$ for the {\rm effective conductance} between $s$ and
$t$. The following result is a slight modification of the results
in\cite[Section~5]{abppcw:ecgd2015}, taking into account the
definition of $\cN_{\text{min}}$ in Remark~\ref{rem:remarks} (c).

\begin{theorem}[\cite{abppcw:ecgd2015}]\label{thm:generalize}
  Let $G$ be a weighted graph and $\Ga$ a family of objects on
  $G$. Then the map $ p\mapsto \Mod_{p,\si}(\Ga) $ is continuous for
  $1\leq p\leq \infty$, $p\mapsto \cN_{\text{min}}^p
  \Mod_{p,\si}(\Ga)$ is decreasing, and the map $ p\mapsto
  \left(\Mod_{p,\si}(\Ga)/\si(E)\right)^{1/p} $ is increasing, where
  $\sigma(E) := \sum_{e\in E}\sigma(e)$.  Moreover, \bi
\item{\bf\it For} $\mathbf{p=\infty}$: 
\[
\Mod_\infty(\Ga)=\frac{1}{\ell(\Ga)}.
\]
\item {\bf\it For} ${\mathbf p=1}$,  when $G$ is unweighted and  $\Ga=\Ga(s,t)$ is a connecting family,
\[
\Mod_1(\Ga)=\min\{|C|:\mbox{ $C$ a cut of $\Ga$}\}
\]

\item {\bf\it For} ${\mathbf p=2}$, if $G$ is undirected and $\Ga=\Ga(s,t)$ is a connecting family,
\[
\Mod_2(\Ga)=\cCeff(s,t),
\]
\ei
\end{theorem}
\begin{remark}
An early version of the case $p=2$ is due to Duffin \cite{duffin:jmaa1962}. Our proof in \cite{abppcw:ecgd2015} takes inspiration from a very general result in metric spaces, see \cite{heinonen2001}.
\end{remark}

\begin{example}[Basic Example]\label{ex:basic}
Let $G$ be a graph consisting of $k$ simple paths in parallel, each path taking $\ell$ hops to connect a given vertex  $s$ to a given vertex $t$. Assume also that $G$ is unweighted, that is $\si\equiv 1$. Let $\Ga$ be the family consisting of the $k$ simple paths.  
Then $\ell(\Ga)=\ell$ and the size of the minimum cut is $k$. A simple computation shows that
\[
\Mod_p(\Ga)=\frac{k}{\ell^{p-1}}\quad\mbox{for } 1\le p<\infty,\qquad
\Mod_\infty(\Ga)=\frac{1}{\ell}.
\]
In particular, $\Mod_p(\Ga)$ is continuous in $p$, and $\lim_{p\to\infty}\Mod_p(\Ga)^{1/p}=\Mod_\infty(\Ga)$. 
Intuitively, when $p\approx 1$, $\Mod_p(\Ga)$ is more sensitive to the number of parallel paths, while for
$p>>1$, $\Mod_p(\Ga)$ is more sensitive to short walks. 
\end{example}

\subsection{Advantages of Modulus}

As we saw above in Theorem \ref{thm:generalize}, the concept of
modulus encompasses classical quantities such as shortest path,
minimal cut, and effective conductance. For this reason, modulus has
many advantages. For instance, in order to give effective conductance
a proper interpretation in terms of electrical networks one needs to
consider the Laplacian operator which on undirected graphs is a
symmetric matrix. On directed graphs however the Laplacian ceases to
be symmetric, so the electrical network model breaks down.  The
definition of modulus, however, does not rely on the symmetry of the
Laplacian and, therefore, can still be defined and computed in this
case.

Moreover,  modulus can measure the richness of many  types of families of objects, not just connecting walk families.
Here are some examples of families of objects that can be measured using modulus:
\bi
\item {\bf Spanning Tree Modulus:} All spanning trees of $G$.
\vt
\item {\bf Loop Modulus:} All simple cycles in $G$.
\vt
\item {\bf Long Path Modulus:} All simple paths with at least $L$ hops
\vt
\item {\bf Via Modulus:} All walks that start at $s$, end at $t$, and visit $u$ along the way. 
\ei

For instance, in \cite{spcsa:jcam2016}, we studied the following centrality measure for a node $v$.
\begin{equation*}\label{eq:Centrality}
C\left(v\right) \defeq
\sum_{i=1}^{\diam(G)}\Mod_2\left(v,S_i(v)\right)
\end{equation*}
where $S_i(v)=\{u: d(u,v)=i\}$ is the layer of nodes that are exactly $i$ hops away from $v$.
In order to compare $C(v)$ to other centrality measures, we  simulated many epidemics and then measured the effects of
vaccinating a percentage of the nodes using different centrality measures to pick the vaccinated nodes.

In the same paper  \cite{spcsa:jcam2016}, we also proposed different measures of ``inbetweenness''.
Let $\Ga \left(A,B;c\right)$ represent the family of all walks that start from a set of nodes $A$, visit a node $c$, and end in $B$. We will assume that these walks are not loops, namely they will not be able to end at the starting node. 
We write  $\Mod_p\left(A,B;c\right)$ for $\Mod_p\left(\Ga \left(A,B;c\right)\right)$. For instance, one can compute the ``inbetweenness'' of a node $v$ by computing the average:
\[
\bE_{s,t}\left[\frac{\Mod_p(s,t;v)}{\Mod_p(s,t)}\right]
:=
{|V|\choose 2}^{-1}\sum_{s\ne t}\left[\frac{\Mod_p(s,t;v)}{\Mod_p(s,t)}\right]
\]
Instead, we decided to capture inbetweenness with a single modulus computation by defining
\[BC\left(v\right) \defeq \Mod_2\left(A,A;v\right)\]
where $A$ consists of a small portion of the most central nodes based on the centrality $C(v)$ mentioned above.

Experimentally, these modulus-based measures of centrality perform very well with respect to other measures. A partial explanation of why modulus seems to be an appropriate tool for studying epidemics is provided in \cite{gassp:ieee2015}, where the concept of Epidemic Hitting Time is introduced and is then compared to modulus as well as other measures.

\subsection{Main results}

In this paper we develop the method of Lagrangian duality for $p$-modulus. More precisely, 
\begin{itemize}
\item we develop  Lagrangian duality for $p$-modulus (Section \ref{sec:cvxopt} and \ref{sec:lagrangian}) and 
\vt
\item discuss the notion of minimal subfamilies (Section \ref{sec:beurlingsbf}), which leads to the identification of a set of  important elements of a family $\Ga$ (Theorem \ref{thm:minimal}).
\vt
\item We then work out the case $p=2$ more explicitly  (Section \ref{sec:caseptwo}) and provide a probabilistic interpretation of modulus, the extremal density and the dual problem (Section \ref{sec:prob-interp}).  
\vt
\item 
 Next we discuss the basic algorithm for computing modulus numerically (Section \ref{sec:alg-basics}), and
we give some examples of minimal subfamilies for connecting modulus as well as spanning tree modulus (Section \ref{sec:most-important-walks}).
\end{itemize}

We acknowledge the anonymous referee for many helpful suggestions that have improved the exposition.

\section{Lagrangian duality for $p$-modulus}

\subsection{$p$-Modulus as an ordinary convex program}\label{sec:cvxopt}

To every object $\gamma\in\Gamma$ we associate  a point in  $\R_{\ge 0}^{E}$ via the correspondence
\begin{equation*}
\gamma\mapsto \cN(\gamma) := \left[\cN(\gamma,e_j)\right]_{j=1,\dots,|E|},
\end{equation*}
where we think of $\cN(\ga)$ as a row vector. 
We can define a partial order  by setting 
\[\gamma_1\preceq \gamma_2 \quad\iff\quad \cN(\gamma_1)\preceq \cN(\gamma_2)\]
where the order on vectors is the usual coordinate-wise order.

A density $\rho$ is admissible for $\Gamma$ if and only if $\rho$ belongs to the half-space 
\[
\{\rho'\in\R^{E}: \cN(\ga) \rho' \geq 1\},
\] 
for every $\ga\in\Ga$.
In particular, $\Adm(\Ga)$ is a convex subset of $\R^{E}$.
 Also the energy $\cE_{\si,p}(\rho)$ is a convex function of $\rho$. So computing $\Mod_p(\Ga)$ consists in solving the following standard convex program.
\begin{equation}
  \label{eq:cvx-opt-orig}
  \begin{split}
    \text{\bf minimize}\qquad &\cE_{\si,p}(\rho)\\
    \text{\bf subject to}\qquad & \cN \rho\geq 1
  \end{split}
\end{equation}
Existence of a minimizer follows from compactness of the $p$-norm ball
in $\mathbb{R}^{E}$ and continuity of
$\cE_{\sigma,p}(\cdot)$. Uniqueness holds when $1<p<\infty$ by strict
convexity of the objective function.

Given an admissible density $\rho\in\Adm(\Ga)$ there is a very useful criterion due to Beurling, who developed it  in function theory, see \cite{ahlfors1973} and also \cite{badger:2013finn}, to determine whether $\rho$ is extremal. We will see in Theorem~\ref{thm:minimal}~(ii) that this criterion provides a characterization of extremal densities.

If $\Ga$ is a family of objects, and $\tilde{\Ga}\subset\Ga$ is a subfamily, we write $\cN(\tilde{\Ga})$ for the submatrix of $\cN$ corresponding to the rows that are in $\tilde{\Ga}$.
\begin{theorem}[Beurling's Criterion for Extremality]
\label{thm:beurling}
Let $G$ be a weighted graph, $\Ga$ a family of objects on $G$, and $1<p<\infty$.  

Then, a density $\rho\in \Adm(\Ga)$ is extremal for $\Mod_p(\Ga)$, if there is $\tilde{\Ga}\subset\Ga$ with $\ell_\rho(\ga)=1$ for all $\ga\in \tilde{\Ga}$, i.e., $\cN(\tilde{\Ga}) \rho =\one$, such that: 
\begin{equation}
\label{eq:beurling}
\mbox{whenever $h\in\R^E$ and $\cN(\tilde{\Ga}) h\geq 0$, then $\sum_{e\in E}h(e)\rho^{p-1}(e)\si(e)\geq 0$.}
\end{equation}
Furthermore, for such a subfamily $\tilde{\Ga}$ we have
\begin{equation}\label{eq:beurlingmod}
\Mod_{p,\si}(\tilde{\Ga})=\Mod_{p,\si}(\Ga).
\end{equation}
\end{theorem}

Note that $h\in\R^E$ and not $\R_{\ge 0}^E$. 
The proof is very simple and well-known, so we reproduce it here for completeness.

\begin{proof}[Proof of Theorem \ref{thm:beurling}]
Let $\tilde{\rho}$ be an arbitrary admissible density for $\Ga$. Set $h \defeq \tilde{\rho} - \rho$. If $\gamma \in \tilde{\Gamma}$, then 
\[
\ell_h(\ga)=\ell_{\tilde{\rho}}(\ga)-\ell_\rho(\ga)= \ell_{\tilde{\rho}}(\ga)-1 \geq 0.
\]
By assumption (\ref{eq:beurling}), this implies  $\sum_{e\in E} h(e)\rho^{p-1} (e)\si(e) \ge 0$. Hence, \[\sum \tilde{\rho}(e)\rho^{p-1}(e)\si(e) - \sum \rho^p(e)\si(e) \ge 0\] 
So by H\"{o}lder's inequality:
\[
\cE_{p,\si}(\rho) =\sum_{e\in E} \si(e)\rho(e)^p \le \sum_{e\in E} \tilde{\rho}(e)\rho^{p-1}(e)\sigma(e) \le \left(\sum_{e\in E} \tilde{\rho}^p(e)\si(e)\right)^{1/p} \left(\sum_{e\in E}\rho^p(e)\si(e)\right)^{1-1/p}
\]
Dividing both sides $\cE_{p,\si}(\rho)^{1-1/p}$ and raising to the $p$-th power, we get that $\cE_{p,\si}(\rho) \le \cE_{p,\si}(\tilde{\rho})$.

To prove (\ref{eq:beurlingmod}), note that  (\ref{eq:beurling}) only requires  $\cN(\tilde{\Ga})(\tilde{\rho}-\rho)\geq 0$, so the same argument above holds verbatim for $\tilde{\rho}$ an arbitrary admissible density for $\tilde{\Ga}$. Therefore, we have $\cE_{p,\si}(\rho)\leq\cE_{p,\si}(\tilde{\rho})$, and $\Mod_{p,\si}(\Ga)\leq\Mod_{p,\si}(\tilde{\Ga})$. The other direction holds by monotonicity, see Remark \ref{rem:remarks} (b).
\end{proof}

\subsection{Families of walks}

In general families of walks on graphs are infinite. For instance, there are infinitely many walks connecting two distinct vertices $s$ and $t$ in a connected graph. Therefore, in this case, the optimization problem in (\ref{eq:cvx-opt-orig}) is subject to infinitely many constraints (one per walk). However, in \cite{adsgpc} we have shown that one can always pick finitely many constraints so as to have the same convex program. This follows from the following result about integer lattices.
\begin{theorem}[\cite{adsgpc}]\label{thm:essential}
Let $X\subset\mathbb{Z}_{\ge 0}^{n}$.  Then  there exists $ X^*\subseteq X$ such that $X^*$ is finite and such that 
for every $x\in X$ there exists $x^*\in X^*$ such that $x^*\preceq x$.   In particular, every family of walks $\Gamma$ admits a finite subfamily $\Gamma^*\subseteq\Gamma$ such that $\Adm(\Gamma^*)=\Adm(\Gamma)$.
\end{theorem}
Since the subfamily $\Ga^*$ has the same admissible densities as $\Ga$, the minimal energy and the extremal density that solve problem (\ref{eq:cvx-opt-orig}) are unchanged if we replace $\Ga$ with $\Ga^*$.

\begin{definition}
If $\Ga^*\subset \Ga$ is finite and $\Adm(\Ga^*)=\Adm(\Ga)$ we say that $\Ga^*$ is an {\it essential subfamily}. 
\end{definition}
Note that, since the admissible set~\eqref{eq:Adm} does not depend on
$p$ or $\si$, neither does the notion of essential subfamily.  The
following example shows that essential subfamilies can be arbitrarily
large even on a graph with only two edges.

\begin{example}\label{example}
Consider the unweighted path graph with two edges $e_1$ and $e_2$ in series, and let $n=1,3,5,...$ be an odd integer. Consider the family of walks $\Ga_n=\{\ga_{n,k}\}_{k=0}^{(n-1)/2}$ where 
\[
\cN(\ga_{n,k},e_1)=n-2k  \qquad\mbox{and}\qquad\cN(\ga_{n,k},e_2)=f_n(n-2k)
\]
where $f_n(n)=n$ and 
for $k=1,\dots,(n-1)/2$,
where $f_n(n-2k)\defeq n+2\sum_{j=1}^k (j+1)$. Notice that
\begin{equation}\label{eq:recurrence}
f_n(n-2(k+1))=f_n(n-2k)+2(k+2).
\end{equation}
and $2(k+2)$ is even. So, since $f_n(n)=n$ is odd, we have that $f_n(n-2k)$ is also odd for all $k$'s. By a parity argument this shows that there indeed exist walks $\ga_{n,k}$ with multiplicity vector $\cN(\ga_{n,k})$ as defined above.

We now study the half-planes $\left\{\rho: \cN(\ga_{n,k})\rho\geq 1\right\}$. To find how they meet we solve the system
\[
\begin{array}{r}
(1)\\
(2)
\end{array}
\left\{
\begin{array}{lllr}
x(n-2k) &+ & yf_n(n-2k) & =1\\
x(n-2k-2) & + & y\left[f_n(n-2k)+2(k+2)\right] & =1
\end{array}
\right.
\]
where $x=\rho(e_1)$ and $y=\rho(e_2)$.
Replacing $(2)$ by $(2)-(1)$ and making a simple substitution yields the solution $x_k=(k+2)y_k$ and $y_k=[(k+2)(n-2k)+f_n(n-2k)]^{-1}$. In particular, $y_k>0$ and thus $x_k>0$. The line defined by equation $(1)$ becomes more and more horizontal as $k$ varies, since $f_n(n-2k)$ grows faster than $n-2k$ decreases. If we can show that $y_k$ decreases we will have shown that the admissible set $\Adm(\Ga_n)$ is an unbounded polygon with at least $(n+1)/2$ faces. A calculation using (\ref{eq:recurrence}) yields
\[
\frac{1}{y_{k}}-\frac{1}{y_{k-1}}=n-2k>0
\]
for $k=1,\dots,(n-1)/2$. So the only essential subfamily of $\Ga_n$ is itself and the cardinality of $\Ga_n$ can be made arbitrarily large.

However,  $\Mod_p(\Ga_n)=\Mod_p(\tilde{\Ga}_n)$ where
$\tilde{\Ga}_n=\{\ga_{n,0}\}$ consists of the single walk $\ga_{n,0}$ such that
$\cN(\ga_{n,0},e_1)=\cN(\ga_{n,0},e_2)=n$. To verify this is a simple application of Beurling's criterion Theorem \ref{thm:beurling}.
Consider the constant density $\rho(e_j)=1/(2n)$, for $j=1,2$. Then, $\ell_\rho(\ga_{n,0})=1$. Also, for $k=1,\dots,(n-1)/2$,
\[
\ell_\rho(\ga_{n,k})=\frac{1}{2n}\left[n-2k+n+2k+k(k+1)\right]\geq 1
\]
So $\rho$ is admissible for $\Ga_n$. Moreover, if $h:E\rightarrow\R$ satisfies $\ell_h(\ga_{n,0})\geq 0$, then
$nh(e_1)+nh(e_2)\geq 0$, so $h(e_1)+h(e_2)\geq 0$. 
Therefore,
\[
\sum_{e\in E}h(e)\rho^{p-1}(e)=(2n)^{1-p}h(e_1)+(2n)^{1-p}h(e_2)=(2n)^{1-p}\left(h(e_1)+h(e_2)\right)\geq 0.
\]

The subfamily $\tilde{\Ga}_n$ is an
example of a minimal subfamily, as will be defined later in Definition \ref{def:minimal}.
\end{example}

Theorem \ref{thm:essential} implies that without loss of generality, after passing to an essential subfamily, modulus becomes an ordinary convex program. In particular, we can formulate the dual problem. 

\subsection{The Lagrangian for modulus}\label{sec:lagrangian}
For the remainder of the paper we assume that $\Ga$ is a finite family
of objects.  One may introduce the dual variables
$\lambda:\Gamma\to\mathbb{R}_{\ge 0}$, that is to say a vector
$\la\in\R^{\Ga}_{\ge 0}$, and consider the Lagrangian for the minimization
problem (\ref{eq:cvx-opt-orig}):
\begin{equation*}
  \label{eq:lagrangian}
  L(\rho,\lambda):= \sum_{e\in E}\sigma(e)\left|\rho(e)\right|^p
  + \sum_{\gamma\in\Ga}\lambda(\gamma)\left(1-\sum_{e\in E}\cN(\gamma,e)\rho(e)\right)
\end{equation*}

Fixing $\rho\in\R^{E}$ and maximizing in $\la\in[0,+\infty)^{\Ga}$ gives
\[
f(\rho)\defeq\sup_{\la\geq 0} L(\rho,\la)=\left\{\begin{array}{cl}
{\displaystyle \sum_{e\in E}\sigma(e)\left|\rho(e)\right|^p} & \qquad \mbox{for $\rho\in \Adm(\Ga)$}\\
\infty & \qquad \mbox{for $\rho\not\in \Adm(\Ga)$}
\end{array}     \right.
\]
Therefore, with the primal problem we recover the modulus problem:
\[
\Mod_p(\Ga)=p^*\defeq\inf_{\rho\in\R^{E}} f(\rho).
\]
Note that the latter minimization is unconstrained. 

We now apply the general fact that:
if $F:X\times Y\to\R$, then 
\[
\Inf_{x\in X}\sup_{y\in Y} F(x,y)\geq \sup_{y\in Y}\Inf_{x\in X} F(x,y)
\]

Fixing $\la\in[0,+\infty)^{\Ga}$ and minimizing in $\rho\in\R^{E}$ gives
\[
g(\la)\defeq\inf_{\rho} L(\rho,\la)
\]
Then, the primal problem is bounded below by the dual problem (this is called {\it weak duality}):
\[
p^*\geq d^*\defeq \sup_{\la\geq0}\Inf_\rho L(\rho,\la)=\sup_{\la\geq 0}g(\la)
\]
\begin{definition}
{\it Strong duality} occurs if there exists $\rho^*$ and $\la^*\geq 0$ such that the following saddle-point property holds.
\begin{equation}\label{eq:strongduality}
L(\rho^*,\la)\leq L(\rho^*,\la^*)\leq L(\rho,\la^*) 
\end{equation}
for every $\rho$ and every $\lambda\geq 0$.
\end{definition}
If strong duality holds, then, for $\rho^*$ and $\lambda^*$ satisfying~\eqref{eq:strongduality},
\[
p^*\leq f(\rho^*)=\sup_{\la\geq 0} L(\rho^*,\la)=L(\rho^*,\la^*)= \inf_{\rho}  L(\rho,\la^*)=g(\la^*)   \leq d^*
\]
Therefore in this case $p^*=d^*$, that is the primal and dual problem coincide. 

There are well-known sufficient conditions for strong duality to hold.
One condition that applies in our convex case is  Slater's condition (see also \cite[Theorem 28.2]{Rockafellar1970}). 
\begin{fact}[Slater]
Strong duality holds if the primal feasible region has an interior point.
\end{fact}
For $\Mod_p$ it is enough to find $\rho$ such that $\ell_\rho(\ga)>1$
for all $\ga\in\Ga$.  Thus, for families of objects such as walks or
subsets of $E$, it is enough to let $\rho\equiv 2$. More generally, if
$\Ga$ is a finite family of objects that are allowed to make
fractional use of an edge, it's enough to set $\rho\equiv
\cN_{\text{min}}^{-1}$, with $\cN_{\text{min}}$ as defined in
Remark~\ref{rem:remarks} (c).

As a consequence we obtain a new formula for modulus:
\begin{equation}\label{eq:dualformula}
\Mod_p(\Ga)= \sup_{\la\geq0}\Inf_\rho \left\{\sum_{e\in E}\sigma(e)\left|\rho(e)\right|^p
  + \sum_{\gamma\in\Ga}\lambda(\gamma)\left[1-\sum_{e\in E}\cN(\gamma,e)\rho(e)\right]\right\}
\end{equation}

We now rewrite the dual problem using standard convex optimization techniques. 

For a fixed $\lambda$, the Lagrangian $L$ is minimized by a density $\rho_\lambda$
which satisfies the following stationarity condition:
\begin{equation}
  \label{eq:stationarity}
  p\sigma(e)\rho_\lambda(e)\left|\rho_\lambda(e)\right|^{p-2} - \sum_\gamma\lambda(\gamma)\cN(\gamma,e) = 0, 
\end{equation}
for every $e\in E$.

Dual feasibility implies that $\lambda\ge 0$, so (\ref{eq:stationarity}) implies that $\rho_\lambda\ge 0$ as well, and therefore
\begin{equation}
  \label{eq:optimal-rho-lambda}
  \rho_\lambda(e) = \left(\frac{1}{p\sigma(e)}\sum_\gamma\lambda(\gamma)\cN(\gamma,e)\right)^{\frac{1}{p-1}}= \left(\frac{1}{p\sigma(e)}(\cN^T \la)(e)\right)^{\frac{1}{p-1}}
\end{equation}
Substituting (\ref{eq:optimal-rho-lambda}) into (\ref{eq:dualformula}) 
we get the following dual problem:
\begin{equation}\label{eq:dualproblem}
  \begin{split}
    \text{\bf maximize} &\qquad g(\lambda)=\sum_\gamma\lambda(\gamma) -(p-1)\sum_e\sigma(e)\left(\frac{1}{p\sigma(e)}\sum_\gamma\lambda(\gamma)\cN(\gamma,e)\right)^{\frac{p}{p-1}}\\
    \text{\bf subject to} &\qquad \lambda(\gamma)\ge 0\qquad\forall\gamma\in\Gamma.
  \end{split}
\end{equation}
\begin{remark}
Note that $\la$ appears in (\ref{eq:dualproblem}) composed with a linear operation, which may have a kernel. So in general the solutions to the dual problem are not unique.  
However, if $\lambda^*$ is a solution to (\ref{eq:dualproblem}), then (\ref{eq:optimal-rho-lambda}) gives a solution $\rho^*$ to
 the primal problem (\ref{eq:cvx-opt-orig}) written in terms of $\la^*$ (and we already know that $\rho^*$ is unique).

\end{remark}

\section{Minimal subfamilies}\label{sec:beurlingsbf}

First, we define Beurling and Lagrangian subfamilies.

\begin{definition}\label{def:beurlingsf}
Let $\Ga$ be a family of objects on a weighted graph $G$ and $p\in (1,\infty)$. We say that a subfamily $\tilde{\Ga}\subset\Ga$ is a {\it Beurling subfamily of $\Ga$ for $\Mod_{p,\si}(\Ga)$}, if the unique extremal density $\rho^*$ yields $\ell_{\rho^*}(\ga)=1$ for all $\ga\in\tilde{\Ga}$ and property (\ref{eq:beurling}) holds for $\rho^*$. In particular, by (\ref{eq:beurlingmod}), we always have $\Mod_{p,\si}(\tilde{\Ga})=\Mod_{p,\si}(\Ga)$.
\end{definition}

\begin{definition}
Let $p\in (1,\infty)$ and let $\lambda^*$  be optimal for the dual problem {\rm (\ref{eq:dualproblem})}. We say that
  \begin{equation}\label{eq:Gammalambdastar}
    \Gamma_{\lambda^*} \defeq \left\{\gamma\in \Ga: \lambda^*(\gamma)>0\right\}  
  \end{equation}
is the {\it Lagrangian subfamily} of $\Ga$ associated to $\lambda^*$.  
\end{definition}

Since problem (\ref{eq:cvx-opt-orig}) is convex, sufficiently smooth,
and exhibits strong duality, the Karush-Kuhn-Tucker (KKT) conditions
provide necessary and sufficient conditions for optimality.  In order
to formulate the KKT conditions for the present problem, enumerate the
edges in $E$ as $\{e_j\}_{j=1}^m$ and the objects in $\Ga$ as
$\{\ga_i\}_{i=1}^k$ and let $\la_i=\la(\ga_i)$. Stated for problem
(\ref{eq:cvx-opt-orig}), the KKT conditions ensure the existence of an
optimal $\rho^*\in\R_{\ge 0}^m$ and dual optimal
$\lambda^*\in\mathbb{R}^k_{\ge 0}$~\cite[Theorem
28.3]{Rockafellar1970} satisfying
\begin{align}
  \lambda^*_i \ge 0,\quad 1 - \ell_{\rho^*}(\gamma_i) \le 0
  \quad  \text{for }i=1,2,\ldots,k & \qquad\mbox{\bf (Dual and Primal Feasibility)} \\
  \lambda^*_i\left(1 - \ell_{\rho^*}(\gamma_i)\right) = 0\quad
  \text{for }i=1,2,\ldots,k & \qquad\mbox{\bf (Complementary Slackness)}\label{eq:comp-slack}\\
  \nabla_\rho L(\rho^*,\lambda^*) = 0 & \qquad\mbox{\bf (Stationarity)}\label{eq:kktstationarity}
\end{align}

Here the density $\rho^*$ is the unique minimizer of~\eqref{eq:cvx-opt-orig}, while $\lambda^*$ is a (possibly non-unique) maximizer for~\eqref{eq:dualproblem}.

\begin{proposition}\label{prop:bcrit}
Let $p\in (1,\infty)$ and let $(\rho^*,\lambda^*)$ be a saddle point for the
Lagrangian.  Then the corresponding Lagrangian subfamily $\Gamma_{\lambda^*}$ is a Beurling subfamily, i.e.,
satisfies Beurling's criterion for $\rho^*$. Namely:
\bi
\item[(a)] $\ell_{\rho^*}(\ga)=1$ for all $\ga\in\Ga_{\lambda^*}$ and 
\item[(b)] for every $h:E\rightarrow\R$ with  $\ell_h(\Ga_{\lambda^*})\geq 0$ we have
$\sum_{e\in E}h(e)\rho^*(e)^{p-1}\geq 0$.
\ei
\end{proposition}

\begin{proof}[Proof of Proposition \ref{prop:bcrit}]
We are going to use the fact that $\rho^*$ and $\la^*$ satisfy the KKT conditions mentioned above.

Part (a) follows from complementary slackness, see (\ref{eq:comp-slack}). Namely, at a saddle point, for every $\ga\in\Ga$, either $\lambda^*(\ga)=0$ or $\ell_{\rho^*}(\ga)=1$. 

Part (b) follows from stationarity, see (\ref{eq:kktstationarity}).  The saddle-point property (\ref{eq:strongduality}) requires that for every $\mu\in\R$
  \begin{equation*}
    \begin{split}
      0 &\le L(\rho^*+\mu h,\lambda^*) - L(\rho^*,\lambda^*) \\
      &\leq \cE_{p,\si}(\rho^*+\mu h) + \sum_{\gamma\in\Gamma_{\lambda^*}}\lambda^*(\ga)(\ell_{\rho^*}(\gamma)- \ell_{\rho^*+\mu h}(\gamma))
       - \cE_{p,\si}(\rho^*)\\
       &= p\mu\sum_{e\in E}\si(e)\rho^*(e)^{p-1}h(e) - \mu\sum_{\ga\in\Ga_{\lambda^*}}\lambda^*(\ga)\ell_h(\gamma) + O(\mu^2).
    \end{split}
  \end{equation*}
Therefore, since $\mu$ is arbitrary,
  \begin{equation*}
    p\sum_{e\in E}\si(e)\rho^*(e)^{p-1}h(e) \geq\sum_{\ga\in\Ga_{\lambda^*}}\lambda^*(\ga)\ell_h(\gamma)\geq 0.
  \end{equation*}
\end{proof}

\begin{definition}\label{def:minimal}
  Let $\Ga$ be a family of objects on a weighted graph $G$ and $p\in
  (1,\infty)$. We say that a subfamily $\tilde{\Ga}\subset\Ga$ is a
  {\it minimal subfamily of $\Ga$ for $\Mod_{p,\si}(\Ga)$} if
  $\Mod_{p,\si}(\tilde{\Ga})=\Mod_{p,\si}(\Ga)$, and removing any
  $\ga$ from $\tilde{\Ga}$ results in
  $\Mod_{p,\si}(\tilde{\Ga})<\Mod_{p,\si}(\Ga)$.
\end{definition}
Theorem \ref{thm:minimal} below will show that minimal subfamilies are Lagrangian and hence Beurling, by Proposition \ref{prop:bcrit}. Moreover, minimal subfamilies are ``small'', in the sense that their cardinality is bounded above by $|E|$. In fact, in specific situations, minimal subfamilies might even be much smaller than $|E|$, as Example \ref{example} shows. 

Assume that $\Ga$ is a finite family of objects. By definition, the matrix $\cN=\left[\cN(\ga,e)\right]_{\ga\in\Ga,e\in E}$ has rank at most $\min\{|\Ga|,|E|\}$. 
It will be important in designing algorithms for computing modulus to understand when $\cN$ has full rank.
\begin{theorem}\label{thm:minimal}
Let $p\in (1,\infty)$. Let $\Ga$ be a finite family of objects on a graph $G$. 
Then 
\bi
\item[(i)] A minimal  subfamily $\tilde{\Ga}\subset\Ga$ always exists.
\item[(ii)] If $\tilde{\Ga}$ is a minimal  subfamily for $\Mod_{p,\si}(\Ga)$, then there exists a unique  optimal $\lambda^*$ so that $\tilde{\Ga}=\Ga_{\lambda^*}$. In particular, $\tilde{\Ga}$ is a Beurling subfamily. 
\item[(iii)] The objects in a minimal subfamily $\tilde{\Ga}$  are ``linearly independent'', in the sense that
\[\rank\cN(\tilde{\Ga})=|\tilde{\Ga}|.\]
\item[(iv)] In particular, this implies that the cardinality of minimal subfamilies is bounded above by $|E|$, where $|E|$ is the number of edges.
\item[(v)] Moreover,  if $\Ga$ is itself minimal, then the optimizer for the dual problem is unique.
\ei
\end{theorem}

\begin{proof}[Proof of Theorem \ref{thm:minimal}]
To see (i) note that, by monotonicity (Remark \ref{rem:remarks} (b)), removing objects from $\Ga$, one at a time, leads to a minimal subfamily.

For (ii), first compute $\Mod_{p,\si}(\tilde{\Ga})$. This leads to an extremal density $\rho^*(\tilde{\Ga})$ and at least one set of optimal dual variables $\la^*(\tilde{\Ga})$. Then, since $\Mod_{p,\si}(\tilde{\Ga})=\Mod_{p,\si}(\Ga)$,  $\la^*=\la^*(\tilde{\Ga})$ gives a solution to the dual problem \ref{eq:dualproblem} for $\Ga$. In particular, through (\ref{eq:optimal-rho-lambda}), we see that $\rho^*(\tilde{\Ga})$ is equal to $\rho^*$. Hence, we found a set of optimal dual variables $\la^*$ that is supported on $\tilde{\Ga}$. 
Moreover, if $\la^*(\ga)=0$ for some $\ga\in\tilde{\Ga}$, then $\ga$ can be removed from $\tilde{\Ga}$ without affecting the modulus. This contradicts the minimality of $\tilde{\Ga}$. So $\la^*(\ga)>0$ for every
$\ga\in\tilde{\Ga}$. Now assume that there are two distinct set of optimal dual variables $\la_1^*$ and $\la_2^*$, both supported and strictly positive on $\tilde{\Ga}$. By (\ref{eq:optimal-rho-lambda}) and uniqueness of $\rho^*$, we must have $\cN^T \la_1^*= \cN^T \la_2^*$.
Define
\[
\la_t^*:=t\la_1^*+(1-t)\la_2^*
\]
Then by linearity, $\cN^T \la_t^*=\cN^T \la_1^*= \cN^T \la_2^*$ for all $t\in \R$. So $\la_t^*$ is a solution to the dual problem as long as it is feasible. However, since $\la_1^*\neq\la_2^*$ we are guaranteed that $\la_t^*$ becomes zero in some coordinate for an appropriate choice of $t$. In particular, we can find a feasible $\la_t^*$ that is supported on $\tilde{\Ga}$ but is zero on some $\ga\in\tilde{\Ga}$. But by minimality, $\la^*$ is supported on $\tilde{\Ga}$ it must be strictly positive. This is a contradiction, so we have uniqueness.

To prove (iii),  notice that, by (ii), $\tilde{\Gamma}$  is a Lagrangian subfamily for some
  choice of optimal $\lambda^*$.  
 Let $\tilde{\Gamma}=\{\gamma_1,\gamma_2,\ldots,\gamma_r\}$ be an enumeration of
  $\tilde{\Gamma}$ and suppose $r>|E|$.  Then $\rank\cN(\tilde{\Ga})\le |E|<r$, so (after
  possibly reordering $\tilde{\Gamma}$) there exist real numbers
  $\{c_k\}_{k=1}^{r-1}$ such that
  \begin{equation}\label{eq:lin-dep}
    \cN(\gamma_r,e) = \sum_{k=1}^{r-1}c_k\cN(\gamma_k,e)\qquad\forall e\in E.
  \end{equation}
  We wish to show that the first $r-1$ constraints,
  $\ell_{\rho}(\gamma_k)=1$ for $k=1,2,\ldots,r-1$, imply the last
  constraint, $\ell_{\rho}(\gamma_r)$, and, thus, that $\tilde{\Gamma}$ is
  not minimal for $\Mod_{p,\si}(\Ga)$.

  For any $\rho$, Equation~\eqref{eq:lin-dep} implies that
  \begin{equation}\label{eq:rth-constraint}
    \begin{split}
      \ell_{\rho}(\gamma_r) &= \sum_{e\in E}\cN(\gamma_r,e)\rho(e) =
      \sum_{e\in E}\sum_{k=1}^{r-1}c_k\cN(\gamma_k,e)\rho(e) \\ &=
      \sum_{k=1}^{r-1}c_k\sum_{e\in E}\cN(\gamma_k,e)\rho(e) =
      \sum_{k=1}^{r-1}c_k\ell_{\rho}(\gamma_k).
    \end{split}
  \end{equation}
 By Beurling's Criterion (Proposition \ref{prop:bcrit}), for the extremal density $\rho^*$, we have
  $\ell_{\rho^*}(\gamma_k)=1$ for all $k=1,2,\ldots,r$, which, in
  light of Equation~\eqref{eq:rth-constraint}, implies that
  $\sum_{k=1}^{r-1}c_k=1$.  But this shows that the $r$th constraint
  is redundant, and $\gamma_r$ is not needed in $\tilde{\Gamma}$.

Statement (iv) follows from (iii), simply because $\rank\cN(\tilde{\Ga})\leq\min\{|E|,|\tilde{\Ga}|\}$.

The uniqueness assertion (v) also follows from (iii), because by (\ref{eq:optimal-rho-lambda}) any two optimal $\la^*$ can only differ by an element of the kernel of $\cN(\tilde{\Ga})^T$. But (iii) states that the kernel of $\cN(\tilde{\Ga})^T$ is trivial.
\end{proof}

\section{Minimal families in the  $p=2$ case}\label{sec:caseptwo}

For simplicity, assume that $G$ is unweighted. In the $p=2$ case, problem (\ref{eq:cvx-opt-orig}) becomes the following quadratic programming:
\begin{equation}
  \label{eq:quad-prog}
  \begin{split}
    \text{\bf minimize}\qquad &\rho^T\rho\\
    \text{\bf subject to}\qquad & \cN \rho\geq 1
  \end{split}
\end{equation}
And the dual can be written as:

\begin{equation}\label{eq:quad-dual}
  \begin{split}
    \text{\bf maximize} &\qquad \lambda^T \one-\frac{1}{4}\lambda^T \cN\cN^T \lambda\\
    \text{\bf subject to} &\qquad \lambda\ge 0
  \end{split}
\end{equation}

Assume that $\Ga$ is itself a minimal  family, so that, by Theorem \ref{thm:minimal}, we have $\rank\cN=|\Ga|$ and the dual optimizer is unique. 
\begin{definition}
\label{def:overlap-mtx}
Let the {\sl overlap matrix} for $\Ga$ be the matrix $C\defeq \cN\cN^T$. The entries of $C$ correspond to pairs of walks $(\ga_i,\ga_j)$ and
\[
C(\ga_i,\ga_j)=\sum_{e\in E}\cN(\ga_i,e)\cN(\ga_j,e).
\]
\end{definition}
Take for instance the family of walks in our Basic Example \ref{ex:basic}. There $\Ga$ is a family of $k$ disjoint simple paths of length $\ell$. In this case the overlap matrix is diagonal with $C(\ga_i,\ga_i)\equiv\ell$

Note that $\rank\cN=|\Ga|$ implies that $C=\cN\cN^T$ is invertible. So we get the following result.
\begin{proposition}
Suppose $\Ga$ is a minimal family. Then the dual problem  has a unique solution given by
\[
\la^*=2C^{-1}\mathbf{1}.
\] 
Moreover, the $2$-Modulus of $\Ga$ is
\[
\Mod_2(\Ga)=\mathbf{1}^T C^{-1} \mathbf{1}.
\]
\end{proposition}

\begin{proof}
Denote it by $\la^*$ the unique dual optimizer. Then, by Theorem \ref{thm:minimal} (ii), $\la^*(\ga)>0$ for all $\ga\in\Ga$. Therefore, in order to maximize $g(\la)$ in (\ref{eq:dualproblem}) it's enough to compute the gradient and set it equal to zero (as the constraint $\la\geq 0$ will be satisfied in the end). Recall that
\begin{align*}
g(\la)= \mathbf{1}^T \la-\frac{1}{4}\la^T C \la.
\end{align*}
So we can compute the gradient of  $g$ to be:
\[
\nabla g=\mathbf{1}-\frac{1}{2}C\la.
\]
Since $\Ga$ is minimal, $C$ is invertible and $\nabla g=0$ yields $\la^*=2C^{-1}\mathbf{1}$. 

Moreover,
\[
\Mod_2(\Ga)=g(\la^*)=\mathbf{1}^T C^{-1} \mathbf{1}.
\]
\end{proof}

\begin{example}[A simple example for $p=2$]
\label{ex:simple}
Let $G$ be the unweighted graph depicted in Figure \ref{fig:simpleex}.
The family $\Ga$ consists of two walks $\ga_1=a\ c\ b$ and $\ga_2=a\ c\ d\ b$.
\begin{figure}[h]
\includegraphics[scale=0.5]{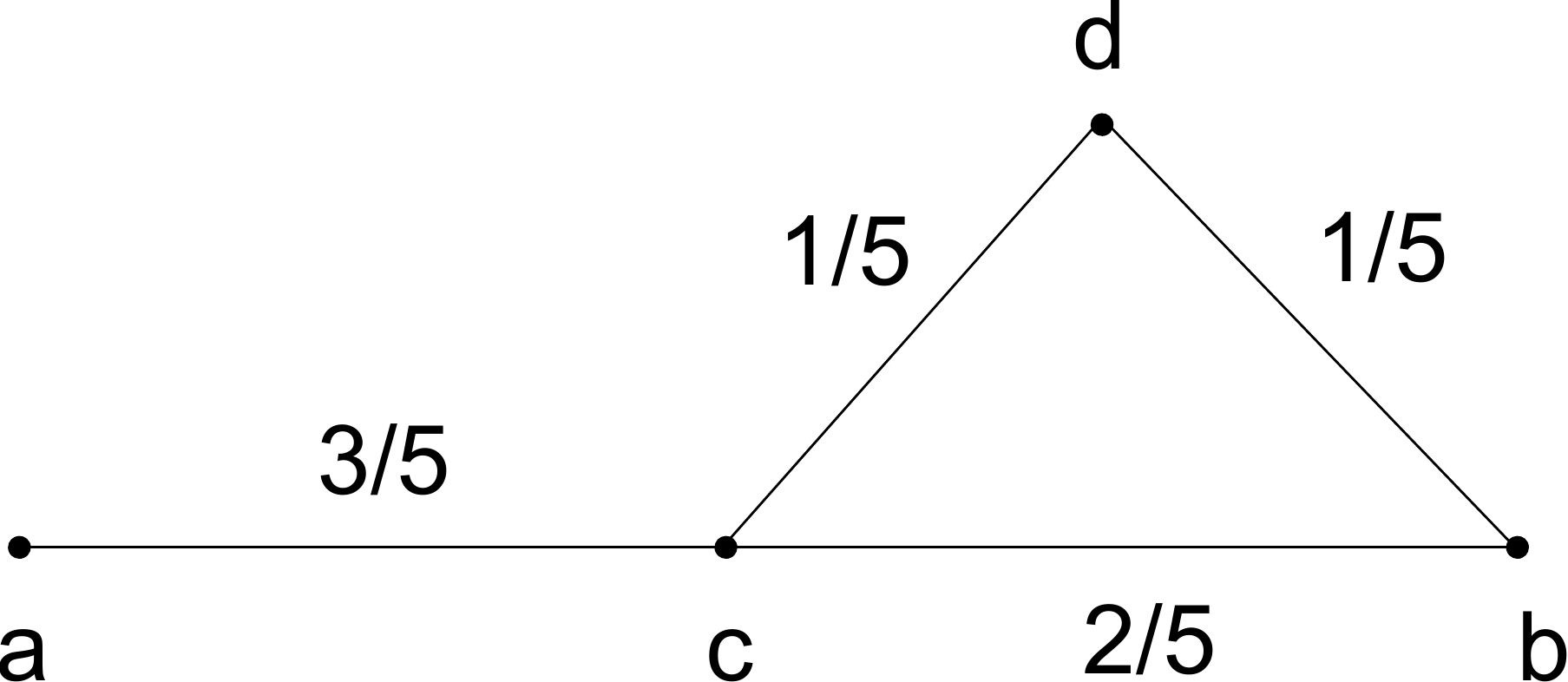}
\caption{A simple example}\label{fig:simpleex}
\end{figure}

 The extremal density $\rho^*$  is given on the edges as in Figure \ref{fig:simpleex}. One can check extremality using Beurling's Criterion (Theorem \ref{thm:beurling}), since $\ell_\rho(\ga_i)=1$ for $i=1,2$. And, whenever $\ell_h(\ga_i)\geq 0$ for $i=1,2$, then
\[
\frac{3}{5}h(ac)+\frac{2}{5}h(cb)+\frac{1}{5}h(cd)+\frac{1}{5}h(db)=\frac{2}{5}\left[h(ac)+h(cb)\right]+\frac{1}{5}\left[h(ac)+h(cd)+h(db)\right]\geq 0
\]

 Observe that if we assign importance $2/5$ to $\ga_1$ and $1/5$ to $\ga_2$, then we would recover $\rho(e)$ by adding the importance of every walk that traverses edge $e$. That is exactly the value of the Lagrange variables in this case. To see this note that in this case 
\[
C=
\left[
\begin{array}{ll}
2 & 1\\
1 & 3
\end{array}\right]
\qquad\mbox{and}
\qquad
C^{-1}=\frac{1}{5}
\left[
\begin{array}{rr}
3 & -1\\
-1 & 2
\end{array}\right].
\]
\[
\qquad
\mbox{So}
\qquad
\frac{\la^*}{2}=C^{-1}\mathbf{1}=\left[\begin{array}{r} 2/5\\1/5  \end{array}\right]
\]
and by (\ref{eq:optimal-rho-lambda}), $\rho^*=\cN^T\la^*/2$.
\end{example}

\section{A probabilistic interpretation}
\label{sec:prob-interp}

The dual in~\eqref{eq:quad-dual} can be reinterpreted in a
probabilistic setting by splitting the nonnegative Lagrange dual
variable $\lambda$ into the product of a probability mass function
(pmf) on $\Gamma$, 
$$\mu\in\cP(\Gamma):=\left\{\mu\in\mathbb{R}_{\ge 0}^\Gamma : \sum_{\gamma\in\Gamma}\mu(\gamma)=1 \right\},$$ 
and a nonnegative scalar
$\nu$.

\begin{theorem}\label{thm:prob}
  Let $\Gamma$ be a finite family of objects on a graph $G$.  Then
  \begin{equation}\label{eq:prob-dual}
    \Mod_2(\Gamma) = \left(\min_{\mu\in\cP(\Gamma)}\mu^T \cN\cN^T \mu\right)^{-1}.
  \end{equation}

Moreover, any optimal pmf $\mu^*$ is related to the extremal density $\rho^*$ for $\Mod_2(\Ga)$ as follows:
\begin{equation}\label{eq:rhomu}
  \frac{\rho^*(e)}{\Mod_2(\Gamma)}=\left(\cN^T\mu^*\right)(e)\qquad\forall e\in E.
\end{equation}

\end{theorem}

\begin{proof}
  The optimization problem in~\eqref{eq:quad-dual} can be written as
\begin{equation*}
  \max_{\mu\in\cP(\Gamma),\nu\ge 0}\left\{
    \nu - \frac{\nu^2}{4}\mu^T \cN\cN^T \mu
  \right\}
  = \max_{\nu\ge 0}\left\{
    \nu - \frac{\nu^2}{4}\min_{\mu\in\cP(\Gamma)}\mu^T \cN\cN^T \mu
  \right\}.
\end{equation*}
By convexity, there exists an optimal choice $\mu=\mu^*$ (not
necessarily unique) for the minimization problem.  Let $\alpha :=
(\mu^*)^T\cN\cN^T\mu^*$ be its energy.  The maximum on the right is
attained when $\nu=\nu^*:=2/\alpha$.  Strong duality then implies that
\begin{equation*}
  \Mod_2(\Gamma) =   \max_{\mu\in\cP(\Gamma),\nu\ge 0}\left\{
    \nu - \frac{\nu^2}{4}\mu^T \cN\cN^T \mu
  \right\}
  = \nu^* - \frac{(\nu^*)^2}{4}(\mu^*)^T\cN\cN^T\mu^*
  = 2/\alpha - 1/\alpha = 1/\alpha,
\end{equation*}
thus showing~\eqref{eq:prob-dual}.
  
Moreover, \eqref{eq:optimal-rho-lambda} establishes (\ref{eq:rhomu}), since
\[
\rho^*(e)= \frac{1}{2}\left(\cN^T\lambda^*\right)(e)
  = \frac{1}{2}\nu^*\left(\cN^T\mu^*\right)(e) = \frac{1}{\alpha}\left(\cN^T\mu^*\right)(e) = \Mod_2(\Ga)\left(\cN^T\mu^*\right)(e).
\]

\end{proof}

In order to obtain a probabilistic interpretation, consider
$\mu\in\cP(\Gamma)$, a pmf on $\Gamma$, which defines a random
variable $\underline{\gamma}$ such that
$\mu(\gamma)=\Pr_{\mu}(\underline{\gamma}=\gamma)$ defines the
probability that $\underline{\gamma}$ takes the value
$\gamma\in\Gamma$.  The values of $\cN^T\mu$ give the expected usage
of each edge by $\underline{\gamma}$.  That is,
\begin{equation}\label{eq:expected-usage}
  \left(\cN^T\mu\right)(e) = \sum_{\ga\in\Ga}\cN(\ga,e)\mu(\ga)=\bE_{\mu}\left[\cN(\underline{\gamma},e)\right].
\end{equation}
Thus, $\rho^*(e)$ is proportional to the expected usage of edge $e$
with respect to a pmf $\mu^*$ which minimizes the energy
$\mu^T\cN\cN^T\mu$.  This energy can be understood by considering two
i.i.d.~random variables $\underline{\gamma}$ and
$\underline{\gamma'}$ with a common pmf $\mu$.  Recalling the overlap matrix
$C$ from Definition~\ref{def:overlap-mtx}, we have
\begin{equation*}
  \begin{split}
    \mu^T\cN\cN^T\mu &=
    \sum_{\gamma,\gamma'\in\Gamma}C(\gamma,\gamma')
    \mu(\gamma)\mu(\gamma') \\
    &= \sum_{\gamma,\gamma'\in\Gamma}C(\gamma,\gamma')
    \Pr_{\mu}(\underline{\gamma}=\gamma, \underline{\gamma'}=\gamma')
    = \bE_{\mu}\left[C(\underline{\gamma},\underline{\gamma'})\right].
  \end{split}
\end{equation*}
In other words, the optimal mass functions $\mu^*$ minimize the
expected overlap of two independent randomly chosen objects in $\Ga$, while the extremal density $\rho^*$ is
proportional to the expected usage of the edges with respect to an optimal probability mass
function.

\begin{example}[Connecting family]
  Considering again Example~\ref{ex:simple}, we see that
  \begin{equation*}
    \frac{\lambda^*}{2} =
    \begin{bmatrix}
      2/5 \\ 1/5
    \end{bmatrix}
    = \frac{3}{5}
    \begin{bmatrix}
      2/3 \\ 1/3
    \end{bmatrix}
    = \frac{3}{5}\mu^*,
  \end{equation*}
  and the modulus can be computed as
  \begin{equation*}
    \sum_{e\in E}\rho^*(e)^2 = \frac{9+1+1+4}{25} = \frac{15}{25} = \frac{3}{5}.
  \end{equation*}

  As Figure \ref{fig:simpleex} indicates, on the edge $ac$, $\rho^*/\Mod_2(\Gamma)=1$, yielding an expected
  usage of $1$, and, indeed, $\cN(\gamma_1,ac)=\cN(\gamma_2,ac)=1$.
  Similarly, edges $cd$ and $db$ have $\rho^*/\Mod_2(\Gamma)$ with a
  value of $1/3$ while the value for $cb$ is $2/3$.  This is
  consistent with $\mu^*$, which gives
  $\Pr_{\mu^*}(\underline{\gamma}=\ga_1)=2/3$ and
  $\Pr_{\mu^*}(\underline{\gamma}=\ga_2)=1/3$.  Finally, since each
  object in $\Gamma$ has a usage of $0$ or $1$ for each edge, the
  expected overlap can be computed as
  \begin{equation*}
    \begin{split}
      \bE_{\mu^*}\left[C(\underline{\gamma},\underline{\gamma'})\right] &=
      2\Pr_{\mu^*}(\underline{\gamma}=\ga_1,\underline{\gamma'}=\ga_1)
      +
      3\Pr_{\mu^*}(\underline{\gamma}=\ga_2,\underline{\gamma'}=\ga_2)\\
      &\quad +
      \Pr_{\mu^*}(\underline{\gamma}=\ga_1,\underline{\gamma'}=\ga_2)
      +
      \Pr_{\mu^*}(\underline{\gamma}=\ga_2,\underline{\gamma'}=\ga_1)\\
      &= 2\cdot\frac{4}{9} + 3\frac{1}{9} + \frac{2}{9} + \frac{2}{9}
      = \frac{5}{3} = \Mod_2(\Gamma)^{-1}.
    \end{split}
  \end{equation*}
\end{example}

An immediate corollary of Theorem~\ref{thm:prob} is the following,
which can be useful for establishing the modulus.
\begin{corollary}\label{cor:prob-bounds}
  Let $\Gamma$ be a finite family of objects on a graph $G$, let
  $\rho\in\Adm(\Gamma)$ and let $\mu\in\cP(\Gamma)$, then
  \begin{equation*}
    \rho^T\rho \ge \Mod_2(\Gamma) \ge \left(\mu^T\cN\cN^T\mu\right)^{-1}.
  \end{equation*}
\end{corollary}

\begin{example}[Spanning Trees]
  Let $G$ be the graph of Figure~\ref{fig:simpleex} and let $\Gamma$
  be the set of spanning trees, with usage function $\cN(\gamma,e)$
  equaling $1$ if $e$ is included in the tree $\gamma$ and $0$ otherwise.
  There are three spanning trees for this graph, and the overlap matrix
  (regardless the enumeration of $\Gamma$) is
  \begin{equation*}
    \cN\cN^T = C =
    \begin{bmatrix}
      3 & 2 & 2 \\ 2 & 3 & 2 \\ 2 & 2 & 3
    \end{bmatrix}.
  \end{equation*}
  Let $\rho$ be the function taking the value $3/7$ on edge $ac$ and
  $2/7$ elsewhere.  This $\rho$ is admissible, since every spanning
  tree must include $ac$ and two additional edges.  Let $\mu$ be the
  uniform pmf on $\Gamma$, then the expected overlap is
  \begin{equation*}
    \left(\frac{1}{3}\one\right)^TC\left(\frac{1}{3}\one\right)
    = \frac{1}{9}\one^TC\one = \frac{21}{9} = \frac{7}{3}.
  \end{equation*}
  By Corollary~\ref{cor:prob-bounds},
  \begin{equation*}
    \frac{3}{7} = \rho^T\rho \ge \Mod_2(\Gamma) 
    \ge \left(\mu^TC\mu\right)^{-1} = \frac{3}{7},
  \end{equation*}
  so $\Mod_2(\Gamma)=\frac{3}{7}$ and, in fact, $\rho$ and $\mu$ are
  optimal for their respective minimization problems.
\end{example}

\begin{remark}
The general dual problem~\eqref{eq:dualproblem} can also be given a probabilistic
interpretation for $1<p<\infty$ and $\sigma$ a set of arbitrary
positive edge weights.  For any $\mu\in\cP(\Gamma)$, let
$\bE_\mu\left[\cN(\underline{\gamma},e)\right]$ be the expected usage
of the edge $e$ with $\underline{\gamma}$ a random variable with pmf
$\mu$.  Then~\eqref{eq:dualproblem} can be reformulated as
\begin{equation*}
  \Mod_p(\Gamma)^{-\frac{1}{p}} = 
  \left(
    \min_{\mu\in\cP(\Gamma)}\sum_{e\in E}\sigma(e)^{-\frac{q}{p}}
    \bE_\mu\left[\cN(\underline{\gamma},e)\right]^q
  \right)^{\frac{1}{q}},
\end{equation*}
with $q=p/(p-1)$ the conjugate H\"{o}lder exponent of $p$.  Thus, the
modulus problem can be reinterpreted as a problem of minimizing a weighted $q$-norm of the expected usage. The unweighted case with $p=2$ is special
in the sense that the sum of the squares of the expected usages can be
reinterpreted as the expected overlap.
\end{remark}

In the case of spanning tree modulus, i.e., when $\Ga_{\rm spt}=\cT$ the family of all spanning trees,  a lower bound can always be obtained from Corollary \ref{cor:prob-bounds} by choosing the uniform distribution $\mu_0$ on $\cP(\cT)$.
The theory of uniform spanning tree is well-developed. For instance the following is a known fact due to Kirchhoff, see \cite[Section 4.2]{lyons-peres}.
\begin{theorem}[Kirchhoff]\label{thm:usteffres}
Let $\mu_0$ be the uniform distribution on $\cT$. Then, given $e=\{x,y\}\in E$,
\[
\bP_{\mu_0}\left[e\in T\right]=\cReff(e)=\Mod_2(x,y)^{-1}
\]
\end{theorem}
Moreover, we can use (\ref{eq:expected-usage}) to rewrite  (\ref{eq:prob-dual})  as follows
\begin{equation}\label{eq:modsumsq}
\Mod_2(\Ga_{\rm spt})^{-1}=\min_{\mu\in\cP(\cT)}\mu^T \cN\cN^T \mu=\min_{\mu\in\cP(\cT)}\|\cN^T \mu\|_2^2=
\min_{\mu\in\cP(\cT)}\sum_{e\in E}\bP_{\mu}\left[e\in T\right]^2
\end{equation}
Therefore, combining this with Corollary \ref{cor:prob-bounds} with obtain a special lower bound for spanning tree modulus.
\begin{corollary}
Let $\Ga_{\rm spt}=\cT$ the family of all spanning trees of a given graph $G$. Then,
\[
\Mod_2(\Ga_{\rm spt})^{-1}\leq \sum_{e\in E}\cReff(e)^2
\]
\end{corollary}

\section{The basic algorithm}
\label{sec:alg-basics}
Our approach in designing algorithms for computing modulus, rather than
focusing on solving (\ref{eq:cvx-opt-orig}) for $\Ga$ or an essential
subfamily $\Ga^*$, will consist in trying to locate a minimal subfamily
$\tilde{\Ga}$, by building an approximating family one walk at the time. This approach allows the algorithm to deal with fairly small families of walks at each step.

The paper~\cite{adsgpc} presents a basic algorithm, suggested by the
monotonicity of $p$-modulus.  In simplest form, the algorithm is as
follows.
\begin{center}
\fbox{
\parbox{10cm}{
\begin{itemize}
\item {\bf Start:} Set $\Ga^\prime=\varnothing$ and $\rho\equiv 0$.

\item  {\bf Repeat:} 
\begin{itemize}
\item If $\ell_\rho(\Ga)\geq 1$, stop.

\item Else find $\ga\in \Ga\setminus\Ga^\prime$ such that $\ell_\rho(\ga)<1$.

\item Add $\ga$ to $\Ga^\prime$.

\item Optimize $\rho$ so that $\cE_p(\rho)=\Mod_p(\Ga^\prime)$.
\end{itemize}
\end{itemize}
}}
\end{center}
\vspace{0.5cm}

This is an example of an exterior point method that repeatedly solves the minimization with a subset of the constraints and adds a violated constraint at each step.  The algorithm will terminate when $\Ga'$ contains a minimal subfamily $\tilde{\Ga}$, since in this case all active constraints will have been added.  However, since $\tilde{\Ga}$ is generally not known a priori, $\Ga'$ may contain many more walks than $\tilde{\Ga}$ upon termination.

The basic algorithm can be improved in a number of ways.  For example, as described in~\cite{adsgpc}, if $\ell_\rho(\Gamma^*)>0$ in any iteration, then $\rho/\ell_\rho(\Gamma^*)\in A(\Gamma^*)$ and, thus, provides an upper bound on $\Mod_p(\Ga^*)$.  This idea allows us to add a stopping condition to the basic algorithm depending on a preset tolerance $\etol$: if $\ell_\rho(\Ga^*)\ge 1-\etol$, then the algorithm terminates.

\begin{theorem}[\cite{adsgpc}] Let $\Gamma$ be a family of walks on a finite graph and suppose that $\rho^*$ is the extremal density for $\Mod_p(\Gamma)$ with $1<p<\infty$. Fix an error tolerance $0<\etol<1$.  Then,  the algorithm with stopping condition $\ell_\rho(\Ga^*)\ge 1-\etol$ will terminate in finite time, and will output a subfamily $\Gamma'\subset\Gamma$ and a density $\rho$ such that
  \begin{equation*}
   \frac{\Mod_p(\Gamma)- \Mod_p(\Gamma')}{\Mod_p(\Ga)} \le \etol,\qquad
\frac{\|\rho^*-\rho\|_p}{\|\rho^*\|_p}\leq \begin{cases}2^{1-1/p} \etol^{1/p} &\qquad \mbox{if $p\geq 2$}\\
\left(\frac{2}{p-1}\etol\right)^{1-1/p}&\qquad\mbox{if $1<p<2$}
\end{cases}
\end{equation*}
\end{theorem}

For the case $p=2$, an efficient primal-dual active set method, based on the algorithm of Goldfarb and Idnani \cite{goldfarb-idnani:mp1983}, can compute the spanning tree modulus (with tolerance $\etol=10^{-3}$) on a graph with approximately 2 million nodes and 2.8 million edges in under 15 minutes.  Details of this algorithm will be presented in a forthcoming paper.

\section{Examples of minimal subfamilies}\label{sec:most-important-walks}

In \cite{abppcw:ecgd2015} we showed that the extremal density $\rho^*$
can be interpreted as sensitivity of modulus to changes in
edge-weights.
\begin{theorem}[\cite{abppcw:ecgd2015}]
 Fix a family $\Ga$ and  let $1<p<\infty$. Define the function
  $F(\sigma):=\Mod_p(\Gamma;\sigma)$ on weights
  $\sigma:E\to(0,\infty)$.  
The function $F$ is Lipschitz continuous and concave.
Moreover, if we let
  $\rho_\sigma^*$ denote the unique extremal density for
  $\Mod_p(\Gamma;\sigma)$,  then 
  \begin{equation}\label{eq:derivative}
    \frac{\partial F(\sigma)}{\partial\sigma(e)} = \rho_\sigma^*(e)^p, \qquad\forall e\in E.
  \end{equation}
\end{theorem}

Equation (\ref{eq:derivative}) shows that the values of the
extremal density, $\rho^*$, can be interpreted as a measure of each
edge's importance in the $p$-modulus problem.  If $\rho^*(e)$ is very
small, then the modulus will not change much if $\sigma(e)$ is
altered.  On the other hand, if $\rho^*(e)$ is very large, then the
modulus will be quite sensitive to changes in $\sigma(e)$.  The dual
formulation of the $p$-modulus given in~\eqref{eq:optimal-rho-lambda} provides a
similarly interesting interpretation.  Each walk $\gamma\in\Gamma$ has
an associated dual variable $\lambda(\gamma)$, which provides a
measure of importance of that walk in the computation of $p$-modulus.
The important walks depend  on $p$, $\si$, and  $\Gamma$. However,
Theorem~\ref{thm:minimal} establishes a uniform bound on their
cardinality.

\begin{example}[A routing example]\label{ex:routing}

Figure~\ref{fig:routing-overlap} shows an example of connecting
modulus with the family $\Gamma$ as the set of paths connecting the
left-most to the right-most node.  The modulus approximation, with
tolerance $\etol=10^{-15}$, is $\Mod(\Gamma)\approx 0.741024$.  The
values $$\rho^*(e)/\Mod(\Gamma) =
\bE_{\mu^*}\left[\cN(\underline{\gamma},e)\right]$$ are shown on each
edge.  The expected overlap,
$\bE_{\mu}\left[C(\underline{\gamma},\underline{\gamma'})\right]$, for
this example is approximately $1/0.741024=1.34948395733$.  The dual
problem~\eqref{eq:prob-dual} yields an optimal pmf $\mu^*$ supported
on 10 paths, shown in Figure~\ref{fig:routing-mu} by thick black
lines.  The values of $\mu^*$ on these paths are shown above each
picture.

\begin{figure}[h!]
  \centering
  \includegraphics[width=0.9\textwidth]{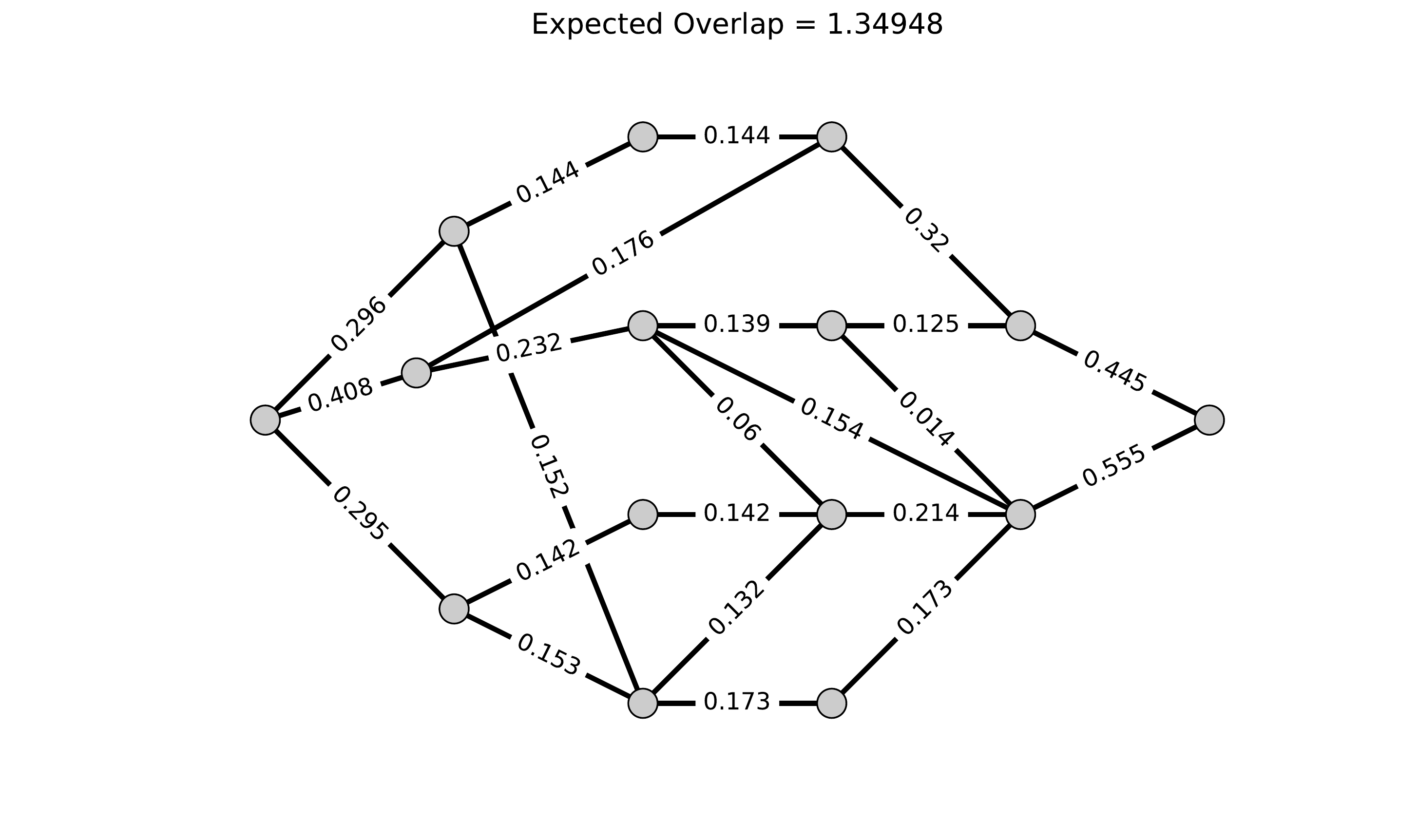}
  \caption{Expected edge usage,
    $\bE_{\mu^*}\left[\cN(\underline{\gamma},e)\right]$, with respect
    to the optimal pmf $\mu^*$ for Example~\ref{ex:routing}.}
  \label{fig:routing-overlap}
\end{figure}

\begin{figure}[h!]
 \includegraphics[width=0.8\textwidth]{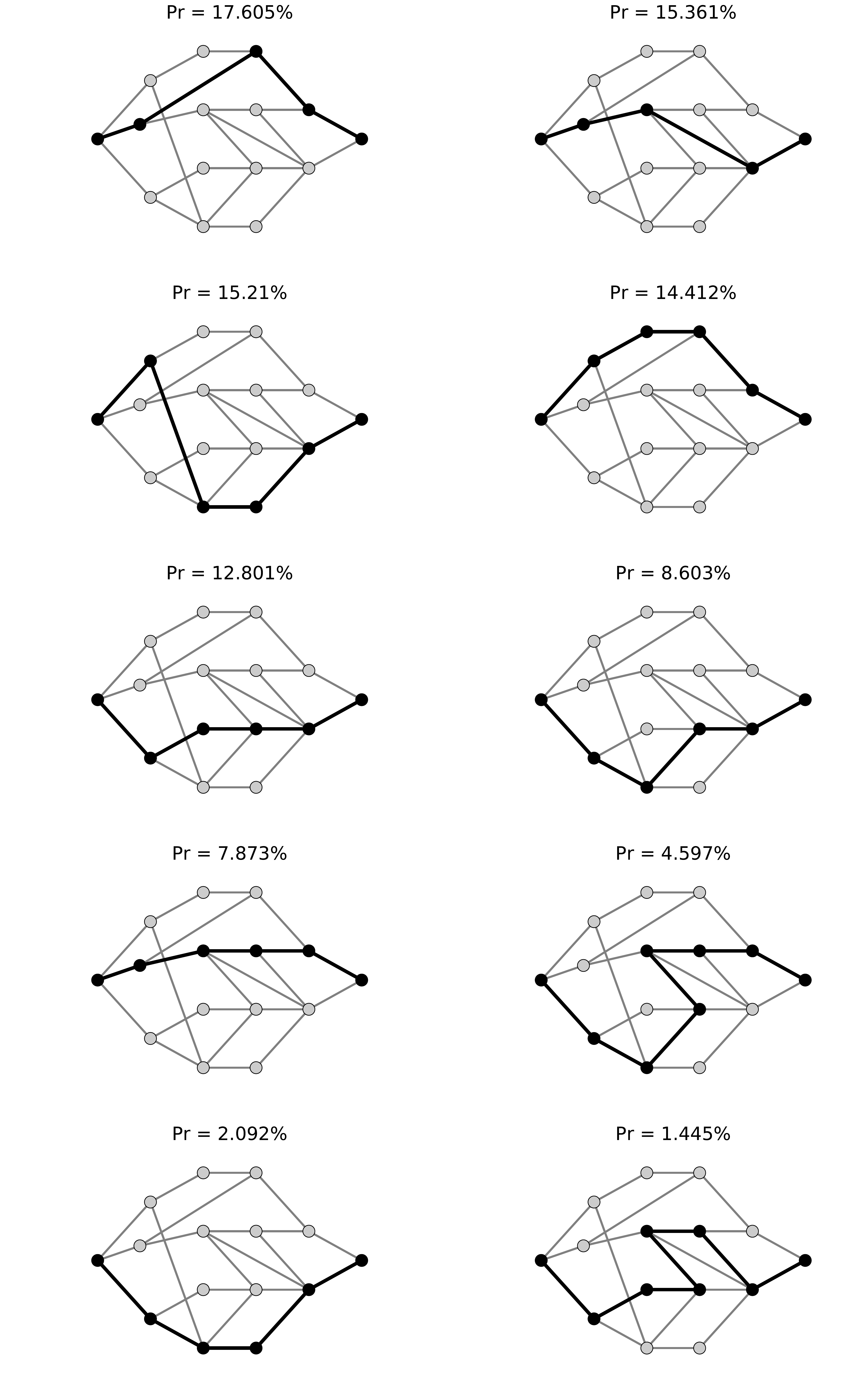}
 \caption{Optimal pmf $\mu^*$ for Example~\ref{ex:routing}.}
 \label{fig:routing-mu}
\end{figure}

\end{example}

\begin{example}[A spanning tree example]\label{ex:spanning}

\begin{figure}
  \centering
  \includegraphics[width=0.9\textwidth]{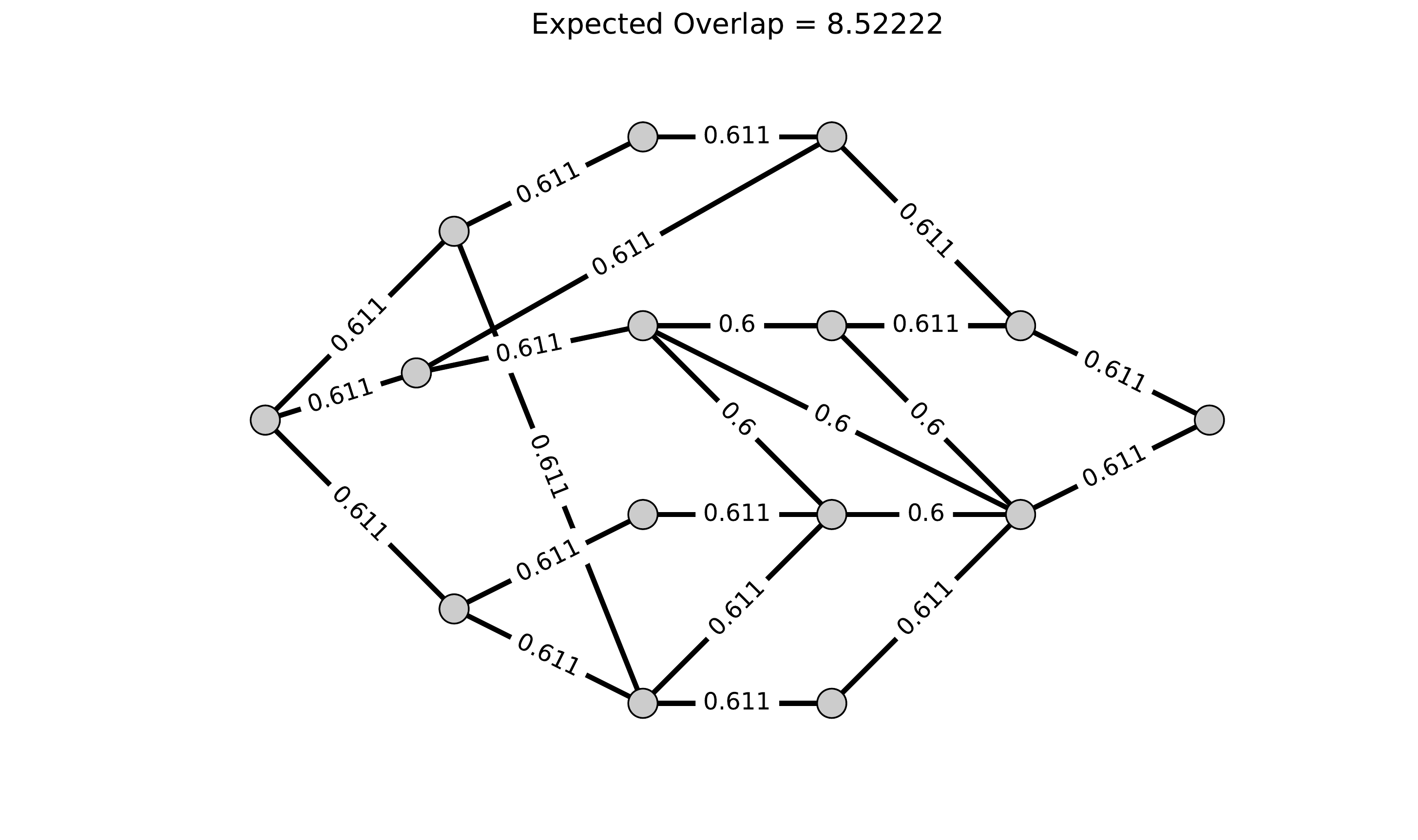}
  \caption{Expected edge usage,
    $\bE_{\mu^*}\left[\cN(\underline{\gamma},e)\right]$, with respect
    to the optimal pmf $\mu^*$ for Example~\ref{ex:spanning}.}
  \label{fig:st-overlap}
\end{figure}

\begin{figure}[h!]
  \centering
  \includegraphics[width=0.9\textwidth]{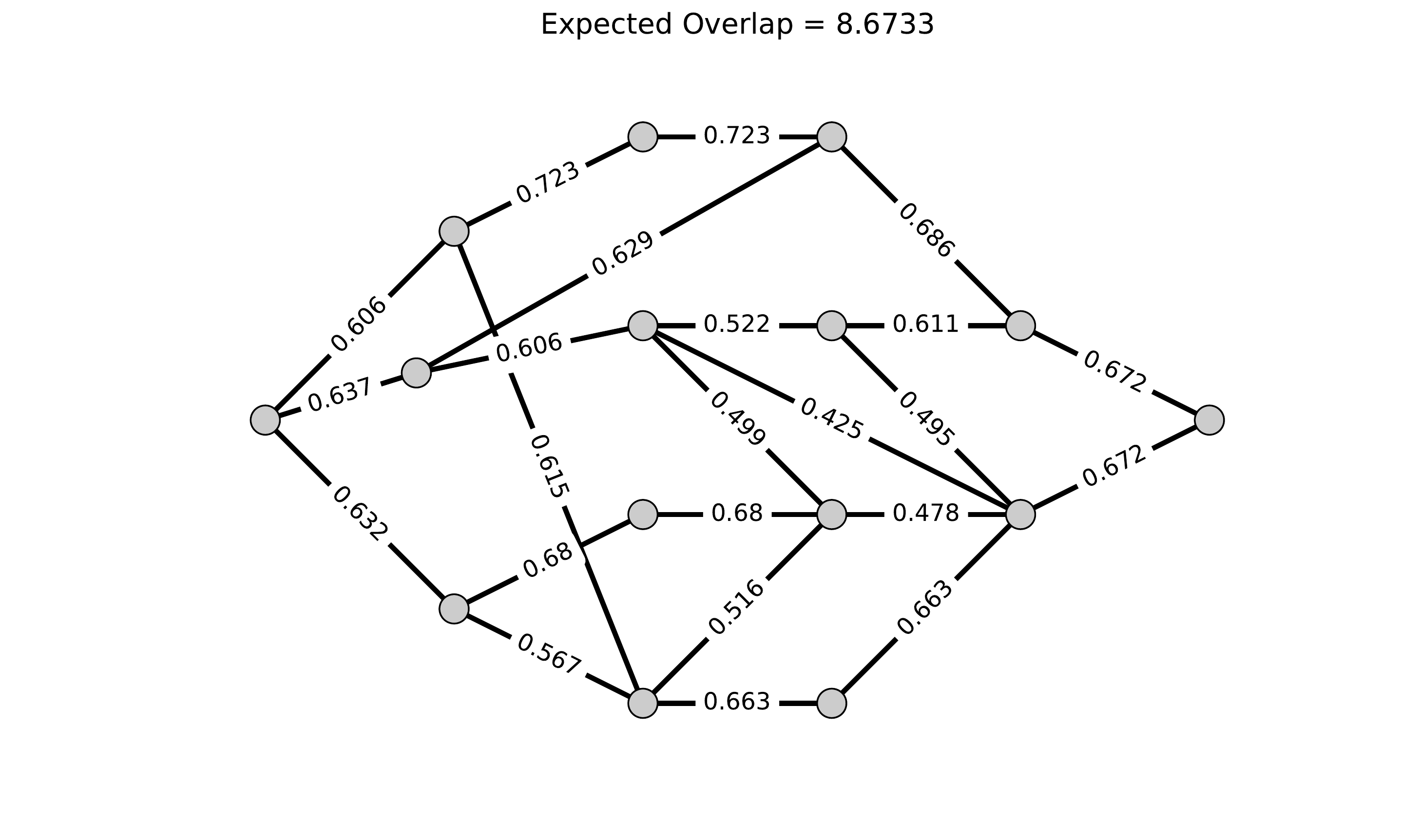}
  \caption{Expected edge usage,
    $\bE_{\mu_0}\left[\cN(\underline{\gamma},e)\right]$, with respect
    to the uniform pmf $\mu_0$ for Example~\ref{ex:spanning}.}
  \label{fig:st-overlap-unif}
\end{figure}

\begin{figure}
 \includegraphics[width=0.8\textwidth]{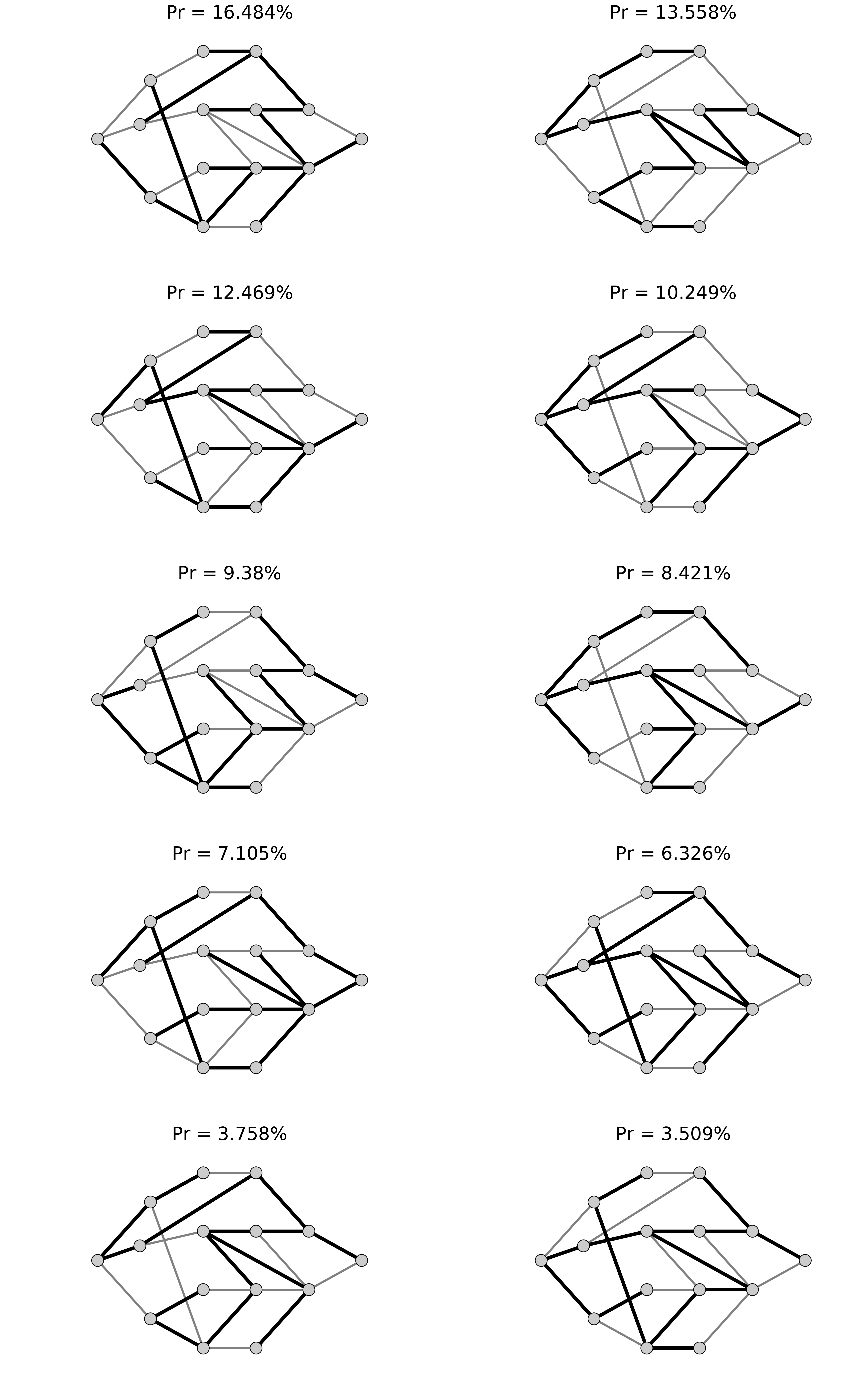}
 \caption{Ten spanning trees with largest value of optimal pmf $\mu^*$
   for Example~\ref{ex:spanning}.}
 \label{fig:st-mu}
\end{figure}

Figure~\ref{fig:st-overlap} shows an example of spanning tree modulus
on the same graph as in Example~\ref{ex:routing}.  The modulus
approximation, with tolerance $\etol=10^{-15}$, is
$\Mod(\Gamma)\approx 0.11734$.  The values $$\rho^*(e)/\Mod(\Gamma) =
\bE_{\mu^*}\left[\cN(\underline{\gamma},e)\right]$$ are shown on each
edge.  Notice that the expected usage is nearly identical on all
edges, taking only two distinct values, $0.6$ and $0.611$.  

The expected overlap,
$\bE_{\mu}\left[C(\underline{\gamma},\underline{\gamma'})\right]$, for
this example is approximately $1/0.11734=8.5222$.  The dual
problem~\eqref{eq:prob-dual} yields an optimal pmf $\mu^*$ supported
on 22 trees.  In Figure~\ref{fig:st-mu}, the thick black lines
indicate the 10 most likely spanning trees, sampled according to
$\mu^*$.  The values of $\mu^*$ on these trees are shown above each
picture.

For comparison, Figure~\ref{fig:st-overlap-unif} shows the expected
usage and expected overlap when the uniform pmf, $\mu_0$, is used.
Since there are 72,650 spanning trees for this graph (as computed by
Kirchhoff's Matrix Tree Theorem, see \cite[Theorem 1.19]{harris-hirst-mossinghoff2008}), it is not easy to compute the usage
matrix $\cN$.  However, by Theorem \ref{thm:usteffres}, the expected usage
$\bE_{\mu_0}\left[\cN(\underline{\gamma},e)\right]$ with respect to
the uniform pmf is equal to the effective resistance of the edge $e$.
Thus, the expected usages can be computed from the pseudoinverse of
the graph Laplacian, and the expected overlap is then found by summing
the squares of the effective resistances, see (\ref{eq:modsumsq}).
\end{example}

\clearpage

%
%\bibliographystyle{acm}
%\bibliography{dualmethodpaper}
%\def\cprime{$'$}

\end{document}